\documentclass{amsart}%
\usepackage{amssymb}
\usepackage{amsfonts}
\usepackage{amsmath}
\usepackage{graphicx}%
\providecommand{\U}[1]{\protect\rule{.1in}{.1in}}
\newtheorem{theorem}{Theorem}
\theoremstyle{plain}

\newtheorem{corollary}{Corollary}

\newtheorem{lemma}{Lemma}

\newtheorem{remark}{Remark}

\numberwithin{equation}{section}
\begin{document}
\title[Existence of Bryc random fields for $q>1.$]{Probabilistic implications of symmetries of $q$-Hermite and Al-Salam-Chihara polynomials.}
\author{Pawe\l \ J. Szab\l owski}
\address{Department of Mathematics and Information Sciences\\
Warsaw University of Technology\\
pl. Politechniki 1, 02-516 Warsaw, Poland\\
 }
\email{pszablowski@elka.pw.edu.pl; pawel.szablowski@gmail.com}
\thanks{This paper is in final form and no version of it will be submitted for
publication elsewhere.}
\date{February, 2008}
\subjclass[2000]{Primary 60J27; Secondary 33C45}
\keywords{Stationary random fields with linear regressions, existence problem,
Al-Salam-Chihara polynomials, $q$-Hermite polynomials, addition formula for
$q-$exponential functions, }

\begin{abstract}
We prove the existence of stationary random fields with linear regressions for
$q>1$ and thus close an open question posed by W. Bryc et al.. We prove this
result by describing a discrete $1$ dimensional conditional distribution and
then checking Chapman-Kolmogorov equation. Support of this distribution
consist of zeros of certain Al-Salam-Chihara polynomials. To find them we
refer to and expose known result concerning addition of $q-$ exponential
function. This leads to generalization of a well known formula $(x+y)^{n}%
=\sum_{i=0}^{n}\binom{n}{k}i^{k}H_{n-k}\left(  x\right)  H_{k}\left(
-iy\right)  ,$ where $H_{k}\left(  x\right)  $ denotes $k-$th Hermite polynomial.

\end{abstract}
\maketitle

\section{Introduction}

We consider the existence problem of stationary random fields with linear
regressions that are defined by two parameters $\left(  \rho,q\right)  $ (Bryc
random fields for brevity) for values of $q>1$ and $0<\left\vert
\rho\right\vert <1.$ Positive answer to that question closes an open question
posed by W. Bryc, W. Matysiak and P. J. Szab\l owski in their paper
\cite{AlSalamChihara}.

To do this we have to find zeros of certain Al-Salam-Chihara polynomials.
Exploiting formula for connection coefficients between Al-Salam-Chihara and
$q-$Hermite polynomials that was given in \cite{AlSalamChihara}, we are able
to relate factorization of Al-Salam-Chihara for specific values of parameters
and factorization of expression of the form:%
\begin{equation}
\sum_{k=0}^{n}%
\genfrac{[}{]}{0pt}{}{n}{k}%
_{q}q^{-k(n-k)/2}h_{n-k}\left(  x|q\right)  h_{k}\left(  y|\frac{1}{q}\right)
, \label{IsmStan_2}%
\end{equation}
where $\left\{  h_{n}\left(  x|q\right)  \right\}  _{n\geq-1}$ are so called
continuous $q-$Hermite polynomials. These expressions appear indirectly in the
proof of so called addition theorem for $q-$exponential functions (see e.g.
\cite{ContemporartyMath2000}). An interesting fact about factorization of
expression (\ref{IsmStan_2}) is that it is a generalization of well known
formula for ordinary Hermite polynomials
\[
\sum_{k=0}^{n}\binom{n}{k}H_{n-k}\left(  x\right)  i^{k}H_{k}\left(
-iy\right)  =\left(  x+y\right)  ^{n}.
\]
Hence this paper establishes a connection between probability and advanced
results in $q-$series theory. The properties discussed in the paper seem to
fit into recent interest in Al-Salam-Chihara $q-$Hermite polynomials (see e.g.
\cite{CHS2007}, \cite{Floreanini97}, \cite{LMP2005}, \cite{IS2003}) in quantum physics.

For the sake of completeness and clarity, the paper is organized as follows.
The following section contains brief recollection of basic facts concerning
Bryc random fields. In the next one we recall basic notation and necessary
facts concerning Al-Salam-Chihara polynomials and $q-$Hermite polynomials. The
following section presents main result of the paper that is existence theorem
of Bryc random fields for $q>1.$ At last final section contains longer proofs.

\section{Known facts on Bryc random fields}

The results on Bryc fields are based on \cite{Bryc1}, \cite{Bryc2} and
\cite{AlSalamChihara}.

\begin{enumerate}
\item Bryc processes are defined by two parameters $q\geq-1$ and $\rho
\in(-1,1)\backslash\left\{  0\right\}  .$ For $q\in<-1,1>$ one dimensional
distributions are uniquely defined and depend only on $q.$ Conditional
distributions are also uniquely defined and depend on $q$ and $\rho.$

\item For $q>1$ one dimensional distributions are not uniquely defined. They
only have known all moments. Conditional distribution does not exists for all
pairs $\left(  q,\rho\right)  .$ Namely they might exist only if $\rho^{2}%
\in\left\{  \frac{1}{q},\frac{1}{q^{2}},\ldots\right\}  .$

\item It is also known (see \cite{AlSalamChihara}) that for $q>1$ and
$\rho^{2}=\frac{1}{q^{m-1}}$ conditional distribution of $X_{n+1}$ given
$X_{n}=y$ is concentrated on zeros $\left\{  \chi_{j}\left(  y,q\right)
\right\}  _{j=1}^{m}$ of Al-Salam-Chihara polynomial $p_{m}\left(
x|y,\rho,q\right)  $ defined below. Moreover the masses $\left\{  \lambda
_{i}\right\}  _{i=1}^{m}$ assigned to these zeros are defined by the
equalities: $\sum_{j=1}^{m}\lambda_{j}=1,$ \newline$\sum_{j=1}^{m}\lambda
_{j}p_{k}\left(  \chi_{j}\left(  y,q\right)  |y,\rho,q\right)  \allowbreak
=\allowbreak0;k=1,\ldots m-1.$

\item It is also known that every existing Bryc field (compare. \cite{Bryc1})
satisfies the following relationship: $\mathbb{E}\left(  H_{k}\left(
X_{n+l}|q\right)  |X_{n}\right)  \allowbreak=\allowbreak\rho^{lk}H_{k}\left(
X_{n}|q\right)  ,$ where $\left\{  H_{n}\left(  x|q\right)  \right\}
_{n\geq-1}$ are so called $q-$Hermite polynomials defined by (\ref{Her}).
\end{enumerate}

Following $3$. and $4$. we see that by uniqueness of the solution of the
system of linear equations we can utter the following observation

\begin{remark}
\label{przezHermity} Let $q>1$ and $\rho=1/q^{\left(  m-1\right)  /2}$. Masses
$\left\{  \lambda_{i}\right\}  _{i=1}^{m}$ of the conditional distribution of
$X_{n+1}$ given $X_{n}=y$ are defined by the equivalent system of linear
equations: $\sum_{j=1}^{m}\lambda_{j}=1,$ $\sum_{j=1}^{m}\lambda_{i}%
H_{k}\left(  \chi_{j}\left(  y,q\right)  |q\right)  \allowbreak=\allowbreak
\rho^{k}H_{k}\left(  y|q\right)  ,$ $k=1,\ldots,m-1$ .
\end{remark}

\section{Facts concerning $q-$Hermite and Al-Salam-Chihara polynomials. Basic
notation and facts}

We adopt notation used traditionally $q-$series theory: $\left(  a;q\right)
_{0}\allowbreak=\allowbreak1,$ $\left(  a;q\right)  _{n}\allowbreak
=\allowbreak%
{\displaystyle\prod\limits_{i=0}^{n-1}}
(1-aq^{i}),$ $\left(  a_{1},\ldots,a_{,k};q\right)  _{n}\allowbreak
=\allowbreak\prod_{i=1}^{k}\left(  a_{i};q\right)  _{n},$ $\left[  0\right]
_{q}\allowbreak=\allowbreak0,$ $\left[  n\right]  _{q}\allowbreak
=\allowbreak1+\ldots q^{n-1},$ $n\geq1,$ $\left[  0\right]  _{q}%
!\allowbreak=\allowbreak1,$ $\left[  n\right]  _{q}!\allowbreak=\allowbreak%
{\displaystyle\prod\limits_{i=1}^{n}}
\left[  n\right]  _{q},$ $%
\genfrac{[}{]}{0pt}{}{n}{k}%
_{q}\allowbreak=\allowbreak\left\{
\begin{array}
[c]{ccc}%
\frac{\left[  n\right]  _{q}!}{\left[  k\right]  _{q}!\left[  n-k\right]
_{q}} & when & 0\leq k\leq n\\
0 & when & k>n
\end{array}
\right.  .$ Notice that we have: $\left(  q,q\right)  _{n}\allowbreak
=\allowbreak\left(  1-q\right)  ^{n}\left[  n\right]  _{n}!,$ $%
\genfrac{[}{]}{0pt}{}{n}{k}%
_{q}\allowbreak=\allowbreak\frac{\left(  q,q\right)  _{n}}{\left(  q,q\right)
_{k}\left(  q,q\right)  _{n-k}}.$ Let us consider also a sequences of
polynomials defined by $3-$term recurrence:
\begin{equation}
xH_{n}\left(  x|q\right)  =H_{n+1}\left(  x|q\right)  +\left[  n\right]
_{q}H_{n-1}\left(  x|q\right)  ;n\geq0,\label{Her}%
\end{equation}%
\begin{equation}
2xh_{n}\left(  x|q\right)  =h_{n+1}\left(  x|q\right)  +\left(  1-q^{n}%
\right)  h_{n-1}\left(  x|q\right)  ;n\geq0,\label{h}%
\end{equation}%
\begin{equation}
p_{n+1}\left(  x|y,\rho,q\right)  =\left(  x-\rho yq^{n}\right)  p_{n}\left(
x|y,\rho,q\right)  -\left(  1-\rho^{2}q^{n-1}\right)  \left[  n\right]
_{q}p_{n-1}\left(  x|y,\rho,q\right)  ;n\geq0,\label{p}%
\end{equation}%
\begin{equation}
B_{n+1}\left(  y|q\right)  =-q^{n}yB_{n}\left(  y|q\right)  +q^{n-1}\left[
n\right]  _{q}B_{n-1}\left(  y|q\right)  ;n\geq0,\label{_b}%
\end{equation}
with $H_{-1}\left(  x|q\right)  \allowbreak=\allowbreak h_{-1}\left(
x|q\right)  \allowbreak=\allowbreak B_{-1}\left(  y|q\right)  \allowbreak
=\allowbreak p_{-1}\left(  x|y,\rho,q\right)  \allowbreak=\allowbreak0~$and
$H_{0}\left(  x|q\right)  \allowbreak=h_{0}\left(  x|q\right)  =\allowbreak
p_{0}\left(  x|y,\rho,q\right)  \allowbreak=\allowbreak B_{0}\left(
y|q\right)  \allowbreak=\allowbreak1,~$ We will call polynomials (\ref{Her})
$q-$Hermite polynomials, (\ref{h}) (following \cite{Berg96}) continuous
$q-$Hermite polynomials. Notice that $H_{n}\left(  x;1\right)  \allowbreak
=\allowbreak H_{n}\left(  x\right)  ,$ where $\left\{  H_{n}\left(  x\right)
\right\}  _{n\geq-1}$ denote ordinary (sometimes called probabilistic) Hermite
polynomials i.e. polynomials orthogonal w.r. to measure $d\mu\left(  x\right)
=\frac{1}{\sqrt{2\pi}}\exp\left(  -x^{2}/2\right)  dx.$ Polynomials (\ref{p})
we will call Al-Salam-Chihara polynomials (see \cite{AlSalamChihara}). All
these polynomials are related to one another by the following relationships:
\begin{equation}
h_{m}\left(  x|q\right)  =\left(  1-q\right)  ^{m/2}H_{m}\left(  2x/\sqrt
{1-q}|q\right)  ;~m\geq-1,\label{H_h}%
\end{equation}
(see \cite{Berg96}).
\begin{equation}
p_{n}\left(  x|y,\rho,q\right)  =\sum_{k=0}^{n}\left[
\begin{array}
[c]{c}%
n\\
k
\end{array}
\right]  _{q}\rho^{n-k}B_{n-k}\left(  y|q\right)  H_{k}\left(  x|q\right)
,\label{rozwiniecie}%
\end{equation}
(see \cite{AlSalamChihara} ) and finally
\begin{equation}
B_{n}\left(  y|q\right)  =\left\{
\begin{array}
[c]{ccc}%
\left(  -1\right)  ^{n\left(  n-1\right)  /2}\left\vert q\right\vert
^{n\left(  n-2\right)  /2}H_{n}\left(  -\sqrt{\left\vert q\right\vert }%
y|\frac{1}{q}\right)   & when & q<0\\
i^{n}q^{n\left(  n-2\right)  /2}H_{n}\left(  i\sqrt{q}y|\frac{1}{q}\right)   &
when & q>0
\end{array}
\right.  \label{B_na_H}%
\end{equation}
(see \cite{AlSalamChihara} ). Finally examining the proof of formula (1.3) in
\cite{ContemporartyMath2000} and performing multiplication of two series we
deduce equality:
\begin{align}
&  \sum_{k=0}^{n}%
\genfrac{[}{]}{0pt}{}{n}{k}%
_{q}q^{-k\left(  n-k\right)  /2}h_{k}\left(  x|q\right)  h_{n-k}\left(
y|\frac{1}{q}\right)  \label{IsmStan}\\
&  =e^{-in\phi}\left(  -q^{\left(  1-n\right)  /2}e^{i\left(  \theta
+\phi\right)  },-q^{\left(  1-n\right)  /2}e^{i\left(  -\theta+\phi\right)
};q\right)  _{n}.\nonumber
\end{align}
where $x=\cos\theta$ and $y=\cos\phi.$

We have simple Lemma:

\begin{lemma}
\label{zamiana}$e^{-in\phi}\left(  -q^{\left(  1-n\right)  /2}e^{i\left(
\theta+\phi\right)  },-q^{\left(  1-n\right)  /2}e^{i\left(  -\theta
+\phi\right)  };q\right)  _{n}\allowbreak=$\newline$2^{n}\allowbreak\left\{
\begin{array}
[c]{ccc}%
\prod_{j=1}^{k}t_{2j-1}\left(  x,y,q\right)  & if & n=2k\\
\prod_{j=0}^{k}t_{2j}\left(  x,y,q\right)  & if & n=2k+1
\end{array}
\right.  ,$ \newline where $t_{n}\left(  x,y,q\right)  \allowbreak=\allowbreak
x^{2}+y^{2}\allowbreak+\allowbreak xy\left(  q^{n/2}+q^{-n/2}\right)
\allowbreak+\allowbreak\left(  q^{n}+q^{-n}-2\right)  /4$ for $n\geq1$ and
$t_{0}\left(  x,y,q\right)  \allowbreak=\allowbreak x+y.$
\end{lemma}

Proof in section \ref{dowody}.

Now applying (\ref{H_h}, \ref{B_na_H}) and (\ref{rozwiniecie}) with
$\rho\allowbreak=\allowbreak\frac{1}{q^{\left(  m-1\right)  /2}}$ we get the
following Corollary

\begin{corollary}
\label{factorization}
\begin{subequations}
\label{uog}%
\begin{gather}
\sum_{k=0}^{m}\left[
\begin{array}
[c]{c}%
m\\
k
\end{array}
\right]  _{q}q^{-\left(  m-1\right)  \left(  m-k\right)  /2}B_{m-k}\left(
y|q\right)  H_{k}\left(  x|q\right)  =\label{uog1}\\
p_{m}\left(  x|y,1/q^{(m-1)/2},q\right)  =\left\{
\begin{array}
[c]{ccc}%
\prod_{j=1}^{i}v_{2j-1}\left(  x,-y,q\right)  & if & m=2i\\
\prod_{j=0}^{i}v_{2j}\left(  x,-y,q\right)  & if & m=2i+1
\end{array}
\right.  ;i\in\mathbb{N}, \label{uog2}%
\end{gather}
where $v_{0}\left(  x,y,q\right)  \allowbreak=\allowbreak x+y,~v_{n}%
(x,y,q)\allowbreak=\allowbreak x^{2}+y^{2}+xy\left(  q^{n/2}+q^{-n/2}\right)
-\frac{q^{n}+q^{-n}-2}{q-1},$ $n\geq1.$
\end{subequations}
\end{corollary}

We also have the following remark.

\begin{remark}
Letting $q->1$ in (\ref{uog}) one gets well known identity for ordinary
Hermite polynomials: $\sum_{i=0}^{m}\binom{m}{k}i^{m-k}H_{m-k}\left(
-iy\right)  H_{k}\left(  x\right)  =\left(  x+y\right)  ^{m},$ since $\forall
k\in\mathbb{N}:B_{k}\left(  y,1\right)  =i^{k}H_{k}\left(  iy;1\right)
\allowbreak=\allowbreak i^{k}H_{k}\left(  iy\right)  $ and $v_{k}\left(
x,y,1\right)  \allowbreak=\allowbreak\left(  x+y\right)  ^{2}.$
\end{remark}

\section{Main results}

Let $\chi_{-n}\left(  y,q\right)  $ and $\chi_{n}\left(  y,q\right)  $ denote
two roots of the equation $v_{n}(x,-y,q)=0,$ for $n\geq1$ and let us put
$\chi_{0}\left(  y,q\right)  \allowbreak=\allowbreak y.$ It is easy to note
that $\chi_{\pm n}\left(  t,q\right)  \allowbreak=\allowbreak t\left(
q^{n/2}+q^{-n/2}\right)  /2\allowbreak+\allowbreak\sqrt{t^{2}+\frac{4}{q-1}%
}\left(  q^{\pm n/2}-q^{-\pm n/2}\right)  /2.$ Let us also define the
following finite subsets of integers for $m\geq1$:
\[
\left(  \mathbf{m}\right)  =\left\{
\begin{array}
[c]{ccc}%
\left\{  0\right\}  & when & m=1\\
\left\{  -\left(  2k-1\right)  ,\ldots,\left(  2k-1\right)  \right\}  & when &
m=2k;k\geq1\\
\left\{  -2\left(  k-1\right)  ,\ldots,0,\ldots,2\left(  k-1\right)  \right\}
& when & m=2k-1;k\geq2
\end{array}
\right.  .
\]
We have the following simple lemma:

\begin{lemma}
\label{dodawanie}$\left(  \mathbf{m}\right)  +\left(  \mathbf{n}\right)
=\left\{  i+j:i\in\left(  \mathbf{m}\right)  ,j\in\left(  \mathbf{n}\right)
\right\}  =\left(  \mathbf{m+n-1}\right)  .$
\end{lemma}

Applying the above mentioned corollary and remarks we see that the set$,$
\newline$\left\{  \chi_{i}\left(  y,q\right)  :i\in\left(  \mathbf{m}\right)
\right\}  $ contain roots of the equation: $0\allowbreak=\allowbreak
p_{m}\left(  x|y,1/q^{(m-1)/2},q\right)  .$

These roots have some special properties, that we will put together in the
following lemma where we denoted $\overset{\_}{\mathbb{Z}}\allowbreak
=\allowbreak\mathbb{Z}\backslash\left\{  0\right\}  .$

\begin{lemma}
\label{uproszczenia}For all $m,n\in\overset{\_}{\mathbb{Z}}$ ; $y\in
\mathbb{C}$ and $q\neq1$we have: i) $\chi_{m}^{2}\left(  y,q\right)
+4/\left(  q-1\right)  \allowbreak=\allowbreak\left(  y\frac{q^{m/2}-q^{-m/2}%
}{2}+\sqrt{y^{2}+\frac{4}{q-1}}\frac{q^{m/2}+q^{-m/2}}{2}\right)  ^{2},$ ii)
$\chi_{m}(\chi_{n}\left(  y,q\right)  ,q)=\chi_{n+m}\left(  y,q\right)  .$
\end{lemma}

Proof in section \ref{dowody}.

Let $\mathbf{S}_{m}\left(  y,q\right)  \allowbreak=\allowbreak\left\{
\chi_{k}\left(  y,q\right)  :k\in\left(  \mathbf{m}\right)  \right\}  .$

\begin{remark}
Notice that $\mathbf{S}_{n}\left(  y,q\right)  \varsubsetneq\mathbf{S}%
_{m}\left(  y,q\right)  $ when $n<m.$
\end{remark}

As it follows from \cite{AlSalamChihara} elements of $\mathbf{S}_{m}\left(
y,q\right)  $ constitute a support of a discrete type distribution
$\Lambda\left(  m,y,q\right)  $ whose masses will be enumerated in the similar
way as the points of support. Thus, say, mass $\lambda_{i,m}\left(
y,q\right)  $ is the mass of $\chi_{i}\left(  y,q\right)  $ in $\Lambda\left(
m,y,q\right)  ,$ $i\in\left(  \mathbf{m}\right)  $ . Besides from
\cite{AlSalamChihara} it follows that masses $\lambda_{i,m}\left(  y,q\right)
$ satisfy also equations: $\sum_{j=1}^{m}\lambda_{j,m}\left(  y,q\right)  =1,$
$\sum_{i\in\left(  \mathbf{m}\right)  }\lambda_{i,m}\left(  y,q\right)
p_{j}\left(  \chi_{i}\left(  y,q\right)  |y,1/q^{(m-1)/2},q\right)  =0,$ for
$j=1,\ldots,m-1$ or equivalently $\sum_{j=1}^{m}\lambda_{j,m}\left(
y,q\right)  =1,$ $\sum_{i\in\left(  \mathbf{m}\right)  }\lambda_{i,m}\left(
y,q\right)  H_{j}\left(  \chi_{i}\left(  y,q\right)  |q\right)  \allowbreak
=\allowbreak\frac{1}{q^{j\left(  m-1\right)  /2}}H_{j}\left(  y|q\right)  $ by
Remark \ref{przezHermity}.

We also have the following observation:

\begin{lemma}
\label{zlozenie}Let $\Lambda\left(  m,y,q\right)  $ be a discrete distribution
concentrated at points $\mathbf{S}_{m}\left(  y,q\right)  $ with masses
$\lambda_{i,m}\left(  y,q\right)  $, $i\in\left(  \mathbf{m}\right)  .$
Suppose that for every point $z\in\mathbf{S}_{m}\left(  y,q\right)  $ we
consider distribution $\Lambda\left(  n,z,q\right)  $ with masses
$\lambda_{j,n}\left(  z,q\right)  ,$ $j\in\left(  \mathbf{n}\right)  .$ Thus
for each point $y$ we define a distribution concentrated at points $\chi
_{i}\left(  \chi_{j}\left(  y,q\right)  ,q\right)  $ with masses $\lambda
_{i}\left(  y,q\right)  \lambda_{j}\left(  \chi_{i}\left(  y,q\right)
,q\right)  $ for $i\in\left(  \mathbf{m}\right)  ,j\in\left(  \mathbf{n}%
\right)  .$ Then this new distribution is concentrated at $n+m-1$ points of
$\mathbf{S}_{n+m-1}\left(  y,q\right)  $ and masses assigned to them are
$\lambda_{k}^{\ast}\left(  y,q\right)  \allowbreak=\allowbreak\sum
_{i\in\left(  \mathbf{m}\right)  ,j\in\left(  \mathbf{n}\right)
;i+j=k}\lambda_{i}\left(  y,q\right)  \lambda_{j}\left(  \chi_{i}\left(
y,q\right)  ,q\right)  ;k\in\mathbf{S}_{n+m-1}.$ Moreover $\lambda_{k}^{\ast
}\left(  y,q\right)  =\lambda_{k,m+n-1}\left(  y,q\right)  ,$ $k\in\left(
\mathbf{n+m-}1\right)  .$
\end{lemma}

Proof in section \ref{dowody}. From this lemma it follows that every
conditional distributions of Bryc process with $q>1$ and $\rho=1/q^{m}$ :
$m\in\mathbb{N}$ satisfies Chapman-Kolmogorov condition. Namely we have the
following Theorem

\begin{theorem}
Suppose we consider family of discrete one step distributions: \newline%
$\left\{  \Lambda\left(  m,y,q\right)  \right\}  _{y\in\mathbb{R}}$ for some
fixed $m\geq2$ and $q>1.$ Then i) $\left\{  \Lambda\left(  km-k+1,y,q\right)
\right\}  _{q>1,y\in\mathbb{R}}$ is a $k-$ step distributions that is $k-$
fold distribution obtained from randomly selecting $z$ from discrete
distribution $\left\{  \Lambda\left(  m,y,q\right)  \right\}  .$ ii) Suppose
$y$ is selected according to $\left\{  \Lambda\left(  km-k+1,x,q\right)
\right\}  _{x\in\mathbb{R}}$ and for each so selected $y$ we select $z$
according to $\left\{  \Lambda\left(  jm-j+1,y,q\right)  \right\}  ,$ then it
is equivalent to select $z$ from distribution $\left\{  \Lambda\left(  \left(
k+j\right)  m-\left(  k+j\right)  +1,x,q\right)  \right\}  $
\end{theorem}

\begin{proof}
i) Using Lemma \ref{zlozenie} and taking $n=m$ we deduce that $2-$fold
distribution of $\left\{  \Lambda\left(  m,y,q\right)  \right\}
_{y\in\mathbb{R}}$ is equal to $\left\{  \Lambda\left(  2m-1,y,q\right)
\right\}  _{y\in\mathbb{R}}.$ Repeating similar arguments we get the assertion.

ii) Using assertion i) and Lemma \ref{zlozenie} once for $km-k+1$ and secondly
for $jm-j+1$ and using Lemma \ref{dodawanie} we get ii).
\end{proof}

Keeping in mind that one dimensional distributions (determined by $q>1)$ of
such process exist (this follows from Berezanskii's result as reported in
addendum 5 (page 26) of \cite{Akhizer}), process is Markov and its conditional
distributions satisfy Chapman-Kolmogorov equation, we deduce that Bryc process
for $q>1$ and $\rho\in\left\{  1/q^{m}:m\in\mathbb{N}\right\}  $ exists.

\section{Proofs\label{dowody}}

\begin{proof}
[Proof of Lemma \ref{zamiana}]Let us denote \newline$A_{n}\left(  \theta
,\phi,q\right)  \allowbreak=\allowbreak e^{-in\phi}\left(  -q^{\left(
1-n\right)  /2}e^{i\left(  \theta+\phi\right)  },-q^{\left(  1-n\right)
/2}e^{i\left(  -\theta+\phi\right)  };q\right)  _{n}.$ By the definition of
Pochhammer symbol we have $A_{n}\left(  \theta,\phi,q\right)  $ $\allowbreak
=\allowbreak e^{-in\phi}\allowbreak\times\allowbreak\prod_{k=0}^{n-1}\left(
1+e^{i\left(  \theta+\phi\right)  }q^{\left(  1+2k-n\right)  /2}\right)
\allowbreak\times\allowbreak(1+e^{i\left(  -\theta+\phi\right)  }q^{\left(
1+2k-n\right)  /2}).$ Inspecting this product we easily note that
$A_{n}\left(  \theta,\phi,q\right)  \allowbreak/\allowbreak A_{n-2}\left(
\theta,\phi,q\right)  \allowbreak=\allowbreak e^{-2i\phi}\left(  1+e^{i\left(
\theta+\phi\right)  }q^{\left(  1-n\right)  /2}\right)  \allowbreak
\times\allowbreak\left(  1+e^{i\left(  \theta+\phi\right)  }q^{\left(
n-1\right)  /2}\right)  \allowbreak\times\allowbreak(1+e^{i\left(
-\theta+\phi\right)  }q^{\left(  1-n\right)  /2})\allowbreak\allowbreak
\times\allowbreak(1+e^{i\left(  -\theta+\phi\right)  }q^{\left(  n-1\right)
/2})\allowbreak$. This ratio consists of $5$ factors. Multiplying then
according to the scheme (1)\{[(2)(3)][(4)(5)]\} we get $A_{n}\left(
\theta,\phi,q\right)  /A_{n-2}\left(  \theta,\phi,q\right)  \allowbreak
=\allowbreak4$\newline$\left(  x^{2}\allowbreak+\allowbreak y^{2}%
\allowbreak+\allowbreak xy\left(  q^{\left(  1-n\right)  /2}\allowbreak
+\allowbreak q^{\left(  n-1\right)  /2}\right)  \allowbreak+\allowbreak
(q^{n-1}\allowbreak+\allowbreak q^{-1+n}-2\right)  /4)\allowbreak
=\allowbreak4t_{n-1}\left(  x,y,q\right)  ,$ where $x=\cos\theta$ and
$y=\cos\phi.$ Keeping in mind that $A_{1}\left(  \theta,\phi,q\right)
\allowbreak=\allowbreak e^{-i\phi}\left(  1+e^{i\left(  \theta+\phi\right)
}\right)  \allowbreak\times\allowbreak\left(  1+e^{i\left(  -\theta
+\phi\right)  }\right)  \allowbreak=\allowbreak2(x+y)$ it is easy notice that
the assertion is true.
\end{proof}

\begin{proof}
[Proof of the Lemma \ref{dodawanie}]Let us first notice that sums $i+j,$
$i\in\left(  \mathbf{m}\right)  $ , $j\in\left(  \mathbf{n}\right)  $ are
either all even (when $\left(  \mathbf{n}\right)  $ and $\left(
\mathbf{m}\right)  $ contain either both even or odd numbers) or all odd (when
one of $\left(  \mathbf{n}\right)  $ and $\left(  \mathbf{m}\right)  $
contains odd and the other even numbers). Thus we easily see that the set
$\left\{  i+j:i\in\left(  \mathbf{m}\right)  ,~j\in\left(  \mathbf{n}\right)
\right\}  $ is of the form of some the sets $\left(  \mathbf{l}\right)  $. To
see how much is $l$ let us examine the maximal value of the sum $i+j,$
$i\in\left(  \mathbf{m}\right)  ,~j\in\left(  \mathbf{n}\right)  .$ By trivial
inspection we see that it is equal to $m+n-1.$
\end{proof}

\begin{proof}
[Proof of Corollary \ref{factorization}]To prove assertion crucial are the
following observations that are simple consequences of (\ref{H_h}) and
(\ref{B_na_H}) .%
\begin{align*}
B_{n}\left(  y|q\right)   &  =i^{n}q^{n\left(  n-1\right)  /2}h_{n}\left(
iy\frac{\sqrt{q-1}}{2}|\frac{1}{q}\right)  /\left(  q-1\right)  ^{n/2},\\
H_{n}\left(  x|q\right)   &  =\left(  -i\right)  ^{n}h_{n}\left(
ix\frac{\sqrt{q-1}}{2}|q\right)  /\left(  q-1\right)  ^{n/2}.
\end{align*}
Now, observing that $h_{n}\left(  -x|q\right)  \allowbreak=\allowbreak\left(
-1\right)  ^{n}h_{n}\left(  x|q\right)  ,$ we have%
\begin{gather*}
p_{n}\left(  x|y,\frac{1}{q^{\left(  n-1\right)  /2}},q\right)  =\sum
_{k=0}^{n}\left[
\begin{array}
[c]{c}%
n\\
k
\end{array}
\right]  _{q}\frac{1}{q^{\left(  n-1\right)  \left(  n-k\right)  /2}}%
B_{n-k}\left(  y|q\right)  H_{k}\left(  x|q\right)  =\frac{i^{n}}{\left(
q-1\right)  ^{n/2}}\\
\times\sum_{k=0}^{n}\left[
\begin{array}
[c]{c}%
n\\
k
\end{array}
\right]  _{q}q^{-\frac{\left(  n-1\right)  \left(  n-k\right)  }{2}%
+\frac{\left(  n-k\right)  \left(  n-k-1\right)  }{2}}h_{n-k}\left(
iy\frac{\sqrt{q-1}}{2}|\frac{1}{q}\right)  h_{k}\left(  -ix\frac{\sqrt{q-1}%
}{2}|q\right) \\
=\frac{i^{n}}{\left(  q-1\right)  ^{n/2}}\sum_{k=0}^{n}\left[
\begin{array}
[c]{c}%
n\\
k
\end{array}
\right]  _{q}q^{-k\left(  n-k\right)  /2}h_{n-k}\left(  iy\frac{\sqrt{q-1}}%
{2}|\frac{1}{q}\right)  h_{k}\left(  -ix\frac{\sqrt{q-1}}{2}|q\right) \\
=\frac{2^{n}i^{n}}{\left(  q-1\right)  ^{n/2}}\left\{
\begin{array}
[c]{ccc}%
\prod_{j=1}^{k}t_{2j-1}\left(  -ix\frac{\sqrt{q-1}}{2},iy\frac{\sqrt{q-1}}%
{2},q\right)  & if & n=2k\\
\prod_{j=0}^{k}t_{2j}\left(  -ix\frac{\sqrt{q-1}}{2},iy\frac{\sqrt{q-1}}%
{2},q\right)  & if & n=2k+1
\end{array}
\right.  .
\end{gather*}
Now $t_{n}\left(  -ix\frac{\sqrt{q-1}}{2},iy\frac{\sqrt{q-1}}{2},q\right)
\allowbreak=\allowbreak-\frac{\left(  q-1\right)  }{4}\left(  -x^{2}%
-y^{2}+xy\left(  q^{n/2}+q^{-n/2}\right)  +\frac{q^{n}+q^{-n}-2}{q-1}\right)
\allowbreak=\allowbreak\left(  -i\right)  ^{n}\frac{\left(  q-1\right)  }%
{4}v_{n}\left(  x,-y,q\right)  $ and $t_{0}\left(  -ix\frac{\sqrt{q-1}}%
{2},iy\frac{\sqrt{q-1}}{2},q\right)  \allowbreak=\allowbreak\left(  -i\right)
\frac{\sqrt{q-1}}{2}\left(  x-y\right)  $ we get the assertion.
\end{proof}

\begin{proof}
[Proof of the Lemma]\ref{uproszczenia}i) We get this assertion by direct
simple calculation. To prove ii) we use i) and proceed as follows:%
\begin{align*}
\chi_{m}\left(  \chi_{n}\left(  y,q\right)  ,q\right)   &  =\chi_{n}\left(
y,q\right)  \frac{q^{m/2}+q^{-m/2}}{2}\allowbreak+\allowbreak\sqrt{\chi
_{n}^{2}\left(  y,q\right)  +\frac{4}{q-1}}\frac{q^{m/2}-q^{-m/2}}{2}\\
&  =y\frac{q^{\left(  m+n\right)  /2}+q^{-\left(  m+n\right)  /2}}{2}%
+\sqrt{y^{2}+\frac{4}{q-1}}\frac{q^{\left(  m+n\right)  /2}-q^{-\left(
m+n\right)  /2}}{2}.
\end{align*}

\end{proof}

\begin{proof}
[Proof of Lemma \ref{zlozenie}]From lemmas \ref{uproszczenia} and
\ref{dodawanie} we know that there are $n+m-1$ points of the form $\chi
_{i}\left(  \chi_{j}\left(  y,q\right)  ,q\right)  $ with $i\in\left(
\mathbf{n}\right)  $ and $j\in\left(  \mathbf{m}\right)  .$ Moreover we know
that these points are the points of the support of new distribution defined in
Lemma \ref{zlozenie}. We will show that distribution concentrated at points of
$\mathbf{S}_{n+m-1}$ with masses $\lambda_{k}^{\ast},$ $k\in\mathbf{S}%
_{n+m-1}$ satisfy the following system of equations: $\sum_{k\in\left(
\mathbf{n+m}-1\right)  }\lambda_{k}^{\ast}\left(  y,q\right)  =1,$ $\sum
_{k\in\left(  \mathbf{n+m}-1\right)  }\lambda_{k}^{\ast}\left(  y,q\right)
H_{r}\left(  \chi_{k}|q\right)  =\frac{1}{q^{r\left(  n-1+m-1\right)  /2}%
}H_{r}\left(  y|q\right)  $ for $r=1,\ldots,m+n-2.$ Using the fact that
$\chi_{k}\left(  y,q\right)  =\chi_{j}\left(  \chi_{i}\left(  y,q\right)
,q\right)  $ for $i+j=k$ and using definition of the masses $\lambda_{k}%
^{\ast}$ we have:%
\begin{align*}
\sum_{k\in\left(  \mathbf{n+m-1}\right)  }\lambda_{k}^{\ast}H_{r}\left(
\chi_{k}|q\right)   &  =\sum_{i\in\left(  \mathbf{m}\right)  }\lambda
_{i}\left(  y,q\right)  \sum_{j\in\left(  \mathbf{n}\right)  }\lambda
_{j}\left(  \chi_{i}\left(  y,q\right)  ,q\right)  H_{r}\left(  \chi
_{j}\left(  \chi_{i}\left(  y,q\right)  |q\right)  \right) \\
&  =\sum_{i\in\left(  \mathbf{m}\right)  }\lambda_{i}\left(  y,q\right)
\frac{1}{q^{r\left(  n-1\right)  /2}}H_{r}\left(  \chi_{i}\left(  y,q\right)
|q\right) \\
&  =\frac{1}{q^{r\left(  n-1\right)  /2}}\frac{1}{q^{r\left(  m-1\right)  /2}%
}H_{r}\left(  y|q\right)  =\frac{1}{q^{r\left(  n+m-2\right)  /2}}H_{r}\left(
y|q\right)  ,
\end{align*}
$r=1,\ldots,m+n-2.$

By Remark \ref{przezHermity} we deduce that masses $\lambda_{k}^{\ast}$
satisfy also relationships :%
\[
\sum_{k\in\mathbf{S}_{n+m-1}}\lambda_{k}^{\ast}\left(  y,q\right)
p_{r}\left(  \chi_{k}\left(  y,q\right)  ,y,1/q^{(n+m-2)/2},q\right)  =0,
\]
showing that indeed $\lambda_{k}^{\ast}=\lambda_{k,m+n-1}\left(  y,q\right)
,$ $k\in\mathbf{S}_{n+m-1}.$
\end{proof}

\end{document}